\documentclass[reqno]{amsart}
\usepackage{amsfonts,amssymb,amsmath,enumerate}

\usepackage[nobysame]{amsrefs}

\usepackage{tikz}

\numberwithin{equation}{section}

\usepackage{enumitem}
\setlist[enumerate]{format=\normalfont}

\theoremstyle{plain}

\newtheorem{lemma}{Lemma}[]
\newtheorem{theorem}[lemma]{Theorem}

\newtheorem{corollary}[lemma]{Corollary}

\theoremstyle{definition}

\theoremstyle{remark} 

\newtheorem*{Claim}{Claim}

\newcommand{\xra}{\xrightarrow}

\def\mod{\mathop{\rm mod}\nolimits}

\def\proj{\mathop{\rm proj}\nolimits}

\newcommand{\pd}{\mathrm{pd}}

\def\uHom{\mathop{\underline{\rm Hom}}\nolimits}

\def\Hom{\mathop{\rm Hom}\nolimits}

\def\End{\mathop{\rm End}\nolimits}
\def\Ext{\mathop{\rm Ext}\nolimits}

\def\Tr{\mathop{\rm Tr}\nolimits}
\def\add{\mathop{\rm add}\nolimits}
\def\Cok{\mathop{\rm Cok}\nolimits}

\def\Im{\mathop{\rm Im}\nolimits}

\def\gldim{\mathop{\rm gldim}\nolimits}

\def\grade{\mathop{\mathrm{grade}}\nolimits}

\begin{document}

\title{Noncommutative resolutions using syzygies}

\author[Dao, Iyama, Iyengar, Takahashi, Wemyss, Yoshino]{Hailong Dao, Osamu Iyama, Srikanth B. Iyengar\\ Ryo Takahashi,  Michael Wemyss and Yuji Yoshino}

\address{Hailong Dao, Department of Mathematics, University of Kansas, Lawrence, KS 66045-7523, USA.}
\email{hdao@ku.edu}

\address{Osamu Iyama, Graduate School of Mathematics, Nagoya University, Chikusaku, Nagoya 464-8602, Japan.}
\email{iyama@math.nagoya-u.ac.jp}

\address{Srikanth B. Iyengar, Department of Mathematics, University of Utah, Salt Lake City, UT 84112, USA.}
\email{iyengar@math.utah.edu}

\address{Ryo Takahashi, Graduate School of Mathematics, Nagoya University, Chikusaku, Nagoya 464-8602, Japan.}
\email{takahashi@math.nagoya-u.ac.jp}

\address{Michael Wemyss: School of Mathematics and Statistics, University of Glasgow, 15 University Gardens, Glasgow, G12 8QW,
UK.}
\email{michael.wemyss@glasgow.ac.uk}

\address{Yuji Yoshino, Department of Mathematics, Faculty of Science, Okayama University, Tsushima-Naka 3-1-1, Okayama, 700-8530, Japan.}
\email{yoshino@math.okayama-u.ac.jp}

\begin{abstract}
Given a noether algebra with a noncommutative resolution,  a general construction of new noncommutative resolutions is given.  As an application, it is proved that any finite length module over a regular local or polynomial ring gives rise, via suitable syzygies, to a noncommutative resolution. 
\end{abstract}
%%16G30 Representations of orders, lattices, algebras over commutative rings
%%14A22 Noncommutative algebraic geometry
%%13D05 Homological dimension (comm ring section)
\maketitle

\setcounter{section}{1}
The focus of this article is on constructing endomorphism rings with finite global dimension.   This problem has arisen in various contexts, including Auslander's theory of representation dimension \cite{Auslander}, Dlab and Ringel's approach to quasi-hereditary algebras in Lie theory \cite{DR, CPS}, Rouquier's dimension of triangulated categories \cite{Rouquier}, cluster tilting modules in Auslander--Reiten theory \cite{Iyama2}, and Van den Bergh's noncommutative crepant resolutions  in birational geometry \cite{Vandenbergh}.

For a noetherian ring $R$ which is not necessarily commutative, and a finitely generated faithful $R$-module $M$, the ring $\End_{R}(M)$ is a 
\emph{noncommutative resolution} (abbreviated to NCR) if its global dimension  is finite; see \cite{DaoIyamaTakahashiVial}.  
When this happens, $M$ is said to \emph{give an NCR of $R$}.
We give a method for constructing new NCRs from a given one.

\begin{theorem}
\label{th:main}
Let $R$ be a noether algebra, and let $M,X\in\mod R$. If $M$ is a $d$-torsionfree generator giving an NCR, and $\gldim\End_R(X)$ is finite, then for any integer $0\le c<\min\{d,\grade_R X\}$, the following statements hold.
\begin{enumerate}
\item The $R$-module $M \oplus\Omega^{c}X$ is a $c$-torsionfree generator.
\item There is an inequality
\[
\gldim\End_R(M\oplus\Omega^{c}X)\le2\gldim\End_R(M)+\gldim\End_R(X)+1.
\]
\end{enumerate}
In particular, $M\oplus\Omega^{c}X$ gives an NCR of $R$.
\end{theorem}

A commutative ring is \emph{equicodimensional} if every maximal ideal has the same height.  Typical examples of equicodimensional regular rings are polynomial rings over a field, and regular local rings.

\begin{corollary}
\label{co:regular}
Let $R$ be an equicodimensional regular ring, and $N$ a finite length $R$-module such that $\gldim\End_R(N)$ is finite.
Given non-negative integers $c_1,\hdots,c_n$ with $c_{i}<\dim R$ for each $i$, the $R$-module $M:=R\oplus\Omega^{c_1}N\oplus\hdots\oplus \Omega^{c_n}N$ satisfies
\[
\gldim \End_{R}(M)
	\leq 2^{n}\dim R + (2^{n}-1) (\gldim \End_{R}(N)+1).
\]
In particular, $M$ gives an NCR of $R$.
\end{corollary}

For any finite length $R$-module $X$, there exists a finite length $R$-module $Y$ such that $\End_R(X\oplus Y)$ has finite global dimension \cite{Iyama}. In the setting of the corollary,  it follows that an NCR can be constructed using any finite length $R$-module.
 
In the definition of noncommutative resolution, it is sometimes required that the module be reflexive \cite{SpV}.  If $\dim R\geq 3$ in the setting of the corollary, then for any finite length $R$-module, by taking all $c_i\geq 2$ it can be ensured that the module giving the NCR is reflexive, but is not free.

\section*{Proofs}\label{proofs}

Throughout, $R$ will be a \emph{noether algebra}, in the sense that it is finitely generated as a module over its centre, and the latter is a noetherian ring.   Thus $R$ is a noetherian ring, and for any $M$ in $\mod R$, the category of finitely generated left $R$-modules, the ring $\End_{R}(M)$ is also a noether algebra, and hence noetherian.

The \emph{grade} of $M\in\mod R$ is defined to be
\[
\grade_{R} M = \inf\{n \mid \Ext^{n}_{R}(M,R)\ne 0\}.
\]
When $R$ is commutative, this is the length of a longest regular sequence in the annihilator of the $R$-module $M$; see, for instance, \cite[Theorem 16.7]{Matsumura86}.

A finitely generated $R$-module $M$ is \emph{$d$-torsionfree}, for some positive integer $d$, if
\[
\Ext_{R}^{i}(\Tr M,R) = 0 \quad\text{for $1\le i\le d$},
\]
where $\Tr M$ be the Auslander transpose of $M$; see~\cite{AuslanderBridger69}.   This is equivalent to the condition that $M$ is the $d$-th syzygy of an $R$-module $N$ satisfying $\Ext_{R}^{i}(N,R)=0$ for $1\le i\le d$; see \cite{AuslanderBridger69}.

Given $R$-modules $X$ and $Y$ we write $\uHom_{R}(X,Y)$ for the quotient of $\Hom_{R}(X,Y)$ by the abelian subgroup of morphisms factoring through projective $R$-modules.

\begin{lemma}
\label{le:lift}
Let $0\to X\to Y\to Z\to0$ be an exact sequence of $R$-modules. If an $R$-module $W$ satisfies $\underline{\Hom}_R(W,Z)=0$, then the following sequence is exact. 
\[
0 \to \Hom_R(W,X) \to \Hom_R(W,Y) \to \Hom_R(W,Z) \to 0
\]
\end{lemma}
\begin{proof}
By hypothesis any morphism $f\colon W\to Z$ factors as $W\to P\xra{f'} Z$, where $P$ is a projective $R$-module, and since $f'$ lifts to $Y$, so does $f$.
\end{proof}

As usual, we write $\Omega X$ for a syzygy of $X$.

\begin{lemma}
\label{le:stable-isos}
Let $X$ and $Y$ be finitely generated $R$-modules.
\begin{enumerate}
\item\label{le:stable-isos 1}
If $\Ext^1_R(X,R)=0$, then there is an isomorphism
\[
\Omega\colon \uHom_R(X,Y)\xra{\ \cong\ }\uHom_R(\Omega X,\Omega Y).
\]
\item\label{le:stable-isos 2} If $0\le c < \grade_{R}X$ and $n\ge 1$, then $\uHom_R(\Omega^{c}X,\Omega^{c+n}Y)=0$.
\end{enumerate}
\end{lemma}
\begin{proof}
Part (1) is clear, and implies part (2) for its hypotheses yields
\[
\uHom_R(\Omega^{c}X,\Omega^{c+n}Y)\cong\uHom_R(X,\Omega^{n}Y)
\]
and the right-hand module is zero as $\Hom_R(X,R)=0$ implies $\Hom_{R}(X,\Omega^{n} Y)=0$, since $\Omega^{n} Y$ is a submodule of a projective $R$-module.
\end{proof}

\begin{proof}[Proof of Theorem \ref{th:main}]
Part (1) is a direct verification. \\ 
For part (2),  set $A:=\End_R(M\oplus\Omega^{c}X)$ and let $e\in A$ be the idempotent corresponding to the direct summand $M$. Then $eAe=\End_R(M)$, so given the inequality 
\[
\gldim A\le\gldim(eAe)+\gldim A/(e)+ \pd_A(A/(e)) + 1
\]
proved in \cite{AuslanderPlatzekTodorov96}*{Theorem~5.4}, it remains to prove the two claims below.

\begin{Claim} 
There is an isomorphism of rings  $A/(e)\cong \End_R(X)$.
\end{Claim}

Indeed, first note that  $A/(e)=\End_R(\Omega^{c}X)/[M]$, where $[M]$ denotes the two-sided ideal of morphisms factoring through $\add M$.
This does not rely on any special properties of $M$ or of $X$.  

Since $\Hom_{R}(X,R)=0$ one obtains the equality below
\[
\End_R(X)=\underline{\End}_R(X)\cong \underline{\End}_R(\Omega^{c}X),
\]
while the isomorphism is obtained by repeated application of Lemma~\ref{le:stable-isos}\eqref{le:stable-isos 1}, noting that $c<\grade_{R}X$. Therefore, to verify the claim, it is enough to prove $\End_R(\Omega^{c}X)/[M]=\underline{\End}_R(\Omega^{c}X)$, that is, any endomorphism of $\Omega^{c}X$ factoring through $\add M$ factors through $\add R$.  

Given morphisms $\Omega^{c}X\xra{f}M\xra{g}\Omega^{c}X$, the morphism $f$ factors through $\add R$ by Lemma~\ref{le:stable-isos}\eqref{le:stable-isos 2}, since $M$ is a $d$-th syzygy module and $d>c$. This completes the proof of the claim.

\begin{Claim}
There is an inequality $\pd_A(A/(e))\le\gldim\End_R(M)$.
\end{Claim}

Set $n:=\gldim\End_{R}(M)$. Then, 
the $\End_R(M)$-module $\Hom_R(M,\Omega^{c}X)$ has a finite projective resolution
\begin{equation}
0\to P_n\to \cdots \to P_0 \to\Hom_R(M,\Omega^{c}X) \to 0.\label{eq:proj}
\end{equation}
As $\Hom_R(M,-)\colon\add_RM\to\proj\End_R(M)$ is an equivalence, there is a sequence
\begin{equation}
\label{eq:exact1}
0\to M_n\xra{f_n } \cdots \xra{f_1} M_0 \xra{f_0} \Omega^{c}X \to 0
\end{equation}
of $R$-modules, with $M_j\in\add M$ for all $j$, such that the induced sequence
\[
0 \to \Hom_R(M,M_n) \to \cdots \to \Hom_R(M,M_0) \to \Hom_R(M,\Omega^{c}X) \to 0
\]
is isomorphic to \eqref{eq:proj}. Since $R\in\add M$, the sequence \eqref{eq:exact1} is exact.

To justify the claim, it suffices to prove that the induced complex
\begin{equation}
\label{eq:exact2}
0\to\Hom_R(\Omega^{c}X,M_n)\to \cdots\to \Hom_R(\Omega^{c} X,M_0)\xra{g} \Hom_R(\Omega^{c} X,\Omega^{c}X)
\end{equation}
obtained from \eqref{eq:exact1} is exact, and $\Cok(g)$ is isomorphic to $\End_R(\Omega^{c}X)/[M]\cong A/(e)$. For, then there is a projective resolution 
\begin{eqnarray*}
0\to\Hom_R(M\oplus\Omega^{c}X,M_n)\to&\cdots&\to \Hom_R(M\oplus\Omega^{c} X,M_0)\\
&&\to\Hom_R(M\oplus\Omega^{c} X,\Omega^{c}X)\to A/(e)\to 0
\end{eqnarray*}
of the $A$-module $A/(e)$, as desired.

By construction, one obtains the exact sequence
\[
\Hom_R(\Omega^{c}X,M_0)\xra{\ g\ }\Hom_R(\Omega^{c}X,\Omega^{c}X) \to \End_{R}(\Omega^{c}X)/[M] \to 0.
\]
This justifies the assertion about $\Cok(g)$. As to the exactness, for each $0\le i\le n$ set $K_i:=\Im(f_{i})$, where $f_{i}$ are the maps in \eqref{eq:exact1}. Then there are exact sequences
\[
0\to K_{i+1}\to M_{i} \to K_i \to 0.
\]
For each $i\geq 1$, using the fact that $M_i$ is $d$-torsionfree, and $K_{0}=\Omega^{c}X$, it follows by induction that $K_{i}$ is a $(c+1)$-st syzygy. Lemma~\ref{le:stable-isos}\eqref{le:stable-isos 2} then yields that $\uHom_R(\Omega^{c}X,K_i)=0$ for $i \ge 1$. By Lemma~\ref{le:lift}, one then obtains an exact sequence
\[
0 \to  \Hom_R(\Omega^{c}X,K_{i+1}) \to \Hom_R(\Omega^{c}X,M_{i}) \to  \Hom_R(\Omega^{c}X,K_{i}) \to  0.
\]
Thus the sequence \eqref{eq:exact2} is exact, as desired. 
\end{proof}

Recall that a commutative ring $R$ is \emph{regular} if it is noetherian and every localization at a prime ideal has finite global dimension.  When $R$ is further equicodimensional, the global dimension of $R$ is finite, since it equals $\dim R$.

\begin{proof}[Proof of Corollary \ref{co:regular}]
Up to Morita equivalence, we can assume that
\[
c_{1}  > c_{2}  > \cdots > c_{n-1} > c_{n}.
\]
Set $M_{0}=R$ and for each integer $1\le j \le n$, set
\[
M_{j}:= R\oplus\Omega^{c_1}N\oplus\cdots\oplus \Omega^{c_j}N.
\] 
We prove, by an induction on $j$, that  $M_{j}$ is $c_j$-torsionfree and that 
\[
\gldim \End_{R}(M_{j})\leq 2^{j}\dim R + (2^{j}-1) (\gldim \End_{R}(N)+1).
\]
The base case $j=0$ is a tautology, for $R$ is regular and hence its global dimension equals $\dim R$. Assume the inequality holds for $j-1$ for some integer $j\ge 1$. 

For the induction step, set $M=M_{j-1}$, so that
\[
M_{j} = M_{j-1}\oplus \Omega^{c_j} N.
\]
Since $R$ is equicodimensional, $\grade_R N = \dim R$ and  $M_{j-1}$ is $c_{j-1}$-torsionfree, Theorem~\ref{th:main} applies to yield that $M_{j}$ is $c_{j}$-torsionfree, and further that
\[
\gldim\End_{R}(M_{j})\leq 2 \gldim \End_{R}(M_{j-1}) + \gldim \End_{R}(N) + 1.
\]
Applying the induction hypothesis gives the desired upper bound for the global dimension of $\End_{R}(M_{j})$.
\end{proof}

\subsection*{Acknowledgements} 
This paper was written during the AIM SQuaRE on Cohen--Macaulay representations and categorical characterizations of singularities.  We thank AIM for funding, and for their kind hospitality.  Dao was further supported by NSA H98230-16-1-0012, Iyama by JSPS Grant-in-Aid for Scientific Research 16H03923, Iyengar by NSF grant DMS 1503044,
Takahashi by JSPS Grant-in-Aid for Scientific Research 16K05098, Wemyss by EPSRC grant EP/K021400/1,
and Yoshino by JSPS Grant-in-Aid for Scientific Research 26287008.

\end{document}